\documentclass[10pt]{article}
\usepackage{amsmath, amsthm}\usepackage{enumerate}
\usepackage{amssymb}\usepackage{bbold}
\usepackage{color}

\newcommand{\rr}{\mathbb R}

\let\text\hbox

\newtheorem{example}{Example}%[section]

\newtheorem{proposition}{Proposition}%[section]
\newtheorem{lemma}{Lemma}%[section]
\newtheorem{corollary}{Corollary}%[section]

\newtheorem{definition}{Definition}%[section]

\begin{document}

\title{$b$-Property of sublattices in vector lattices}\maketitle
\maketitle\author{\centering{{\c{S}afak Alpay$^{1}$, Svetlana Gorokhova $^{2}$\\ 
\small $1$ Department of Mathematics, Middle East Technical University,  06800 Ankara, Turkey, safak@metu.edu.tr\\ 
\small $2$ Southern Mathematical Institute of the Russian Academy of Sciences, 362027 Vladikavkaz, Russia, lanagor71@gmail.com}
\abstract{We study $b$-property of a sublattice (or an order ideal) $F$ of a vector lattice $E$. In particular, 
$b$-property of $E$ in $E^\delta$, the Dedekind completion of $E$, 
$b$-property of $E$ in $E^u$, the universal completion of $E$, and
$b$-property of $E$ in $\hat{E}(\hat{\tau})$, the completion of $E$.}

{\bf{keywords:}} {\rm  vector lattice, universal completion, Dedekind completion, $b$-property, local solid vector lattice}
\vspace{2mm}

\large

\section{Introduction and preliminaries}

Vector lattices considered here are all real and Ar\-chi\-me\-dean. Vector topologies are assumed to be Hausdorff. 

\bigskip

\begin{definition}\label{b-property}{\em
A sublattice  $F$ of a vector lattice $E$ is said to have {\em $b$-property} in $E$, if $x_\alpha$ is  a net in $F^+$ and $0\le x_\alpha\uparrow \le e$
for some $e\in E$, then there exists $f \in F$  with $0\le x_\alpha\uparrow \le f$.}
\end{definition}

\bigskip

Recall that a subset $F$ of $E$ is said to be majorizing in $E$ if, for each $0 < e \in E$, there exists $f\in F$ with $0 \le e \le f$.

A subset $U$ of a vector lattice (VL) is called {\em solid}  if $|u| \le |v|$, $v \in U$, imply $u\in U$. 
A linear topology $\tau$ on a VL $E$ is called {\em locally solid}  if $\tau$ has a base of zero consisting of solid sets.

A locally solid VL $E$ (LSVL) satisfies the {\em Lebesgue property} if $x_\alpha \downarrow 0$ in $E$ implies $x_\alpha \stackrel{\tau}{\to} 0$.

A LSVL $E(\tau)$ satisfies the {\em Fatou property} if $\tau$ has a base of zero  consisting of solid and order closed sets.

A sublattice $F$ in a VL $E$ is {\em regular} if $\inf A$ is the same as in $F$ and $E$ whenever $A\subset F$ whose infimum 
exists in $F$. Ideals are regular in $E$.

$E$ is called {\em $\sigma$-laterally complete}  if the supremum of every disjoint sequence exists in $E^+$ and 
{\em laterally complete}  if supremum of every disjoint subset  in $E^+$ exists in $E$.

\bigskip

\begin{example}\label{step function} {\em \cite[p.198]{AB1} 
Let $X$ be a topological space. A function $f: X \to \rr$ is called a {\em step function} if there exists a collection of mutually disjoint subsets 
$\{V_i\}$ of $X$ such that  $\bigcup_i V_i = X$, $f$ is constant on  each $V_i$, and $f\in C^{\infty}(X)$. 
Let $S^\infty (X)$ be the space of step functions on an extremally disconnected topological space $X$. Then $S^\infty (X)$ is a laterally 
complete VL.} 
\end{example}

\bigskip

Lateral completion $E^\lambda$ of a VL $E$ is defined to be the intersection of all laterally complete vector lattices between 
$E$ and $E^u$.

Universal completion ($\sigma$-universal completion) of a VL $E$ is a laterally ($\sigma$-laterally) and De\-de\-kind complete 
(De\-de\-kind $\sigma$-complete) vector lattice $E^u$ (resp., $E^s$) which contains $E$ as an order dense sublattice. 
Every VL $E$ has a unique universal completion \cite[Theorem 7.21]{AB1}

\bigskip

\begin{example}\label{extended continuous functions} {\em
Let $X$ be an extremally disconnected  topological space. $C^{\infty}(X)$, the space of all extended continuous functions on $X$ 
with the usual algebraic and lattice operations is a universally complete VL. }
\end{example}

\bigskip

A net $(x_\alpha)_{\alpha\in A}$ in a VL $E$ is {\em order convergent} to $x\in E$ if there exists a net $(x_\beta)_{\beta\in B}$,
possibly over a different index set,  such that $x_\beta \downarrow 0$  and, for each $\beta\in B$, there exists  $\alpha_0 \in A$
with $|x_\alpha - x| \le x_\beta$ for all $\alpha \ge \alpha_0$. In this case we write $x_\alpha \stackrel{o}{\to} x$.

A net $x_\alpha$ in  $E$  {\em $uo$-converges} to $x\in E$ if $|x_\alpha - x| \wedge u \stackrel{o}{\to} 0$ for all $u\in E^+$. 
 In this case we write $x_\alpha \stackrel{uo}{\to} x$.

Let $E(\tau)$ be a LSVL. A net $x_\alpha$ in  $E$  is {\em $u\tau$-convergent} to $x\in E$ if $|x_\alpha - x| \wedge u \stackrel{\tau}{\to} 0$ 
for all $u\in E^+$. A net $x_\alpha$ in  $E$  is called  {\em order Cauchy $($$uo$-Cauchy $)$} if the doubly indexed net 
$(x_\alpha, x_{\alpha'})_{(\alpha, \alpha')}$ is order convergent ($uo$-convergent) to zero.  $E(\tau)$ is called {\em $uo$-complete} if
every $uo$-Cauchy net is $uo$-convergent in $E$.

\bigskip

The $b$-property of a VL $E$ was defined in \cite{AAT03} as: a VL $E$ has {\em $b$-property} if every subset $A$ in $E$ which is 
order bounded in $(E^\sim)^\sim$, remains to be order bounded in $E$. Equivalently, a VL $E$ has $b$-property iff
each net $x_\alpha$ in $E$, with $0\le x_\alpha \uparrow x$ for some $x\in (E^\sim)^\sim$, is order bounded in $E$ (\cite{AAT03}).

\bigskip

\begin{example}\label{perfect VL} {\em
Every perfect VL, and therefore every order dual, have the $b$-property. Every reflexive BL and every $KB$-space have $b$-property
\cite{AAT03, AAT06, AEmG, AEr09}. On the other hand, by considering the basis vectors $e_n$ in $c_0$, we see that $c_0$ does not have 
the $b$-property in $l_\infty$.}
\end{example}

\bigskip

Let us note that Fremlin had considered subsets of a VL $E$ that are order bounded in the universal completion $E^u$ of $E$.
He proved that if $E$ is a $\sigma$-Dedekind complete VL then $E$ is $\sigma$-laterally complete iff $E$ has 
the countable $b$-property in $E^u$ \cite[Theorem 7.38]{AB1}.

\bigskip

\begin{example}\label{projection band} {\em
Each projection band $F$ in a vector lattice $E$ has $b$-property in $E$. In particular, 
every band in a Dedekind complete vector lattice has $b$-property. An element $u$ in a VL $E$ is called an {\em atom} if whenever
$v\wedge w =0$, $0 \le v\le u$, and $0\le w\le u$ imply either $v=0$ or $w=0$. If $x$ is an atom in $E$, the principal band $B_x$ 
generated by $x$ is a projection band and therefore has $b$-property in $E$.}
\end{example}

\bigskip

\begin{example}\label{majorizing sublattice} {\em
Every majorizing sublattice $F$ has $b$-property in $E$. Let $0\le x_\alpha\uparrow \le e$ for some net $x_\alpha \subseteq F$, $e\in E$.
As $F$ is majorizing, there exists $f\in F$ with $e\le f$. Then $0 \le x_\alpha \le f$. 
Since it is well-known that $E$ is majoring in $E^\delta$, $E$ has $b$-property in $E^\delta$.}
\end{example}

\bigskip

\begin{example}\label{order ideal} {\em
Every order ideal $F$ in a vector lattice $E$ with $b$-property in $E$ is a band of $E$. Indeed,  let $x_\alpha$ be a net in $F$ such that
$0 \le x_\alpha\uparrow e \in E$, then $x_\alpha$ is $b$-bounded in $E$ and by the $b$-property of $F$, 
there exists $f\in F$ with $0 \le x_\alpha \le f$. As $x_\alpha\uparrow e$, we have $0 \le e \le f$ and as $F$ is an ideal, $e\in F$.}
\end{example}

\bigskip

\begin{example}\label{ideal I(F)} {\em
Let $F \subseteq E$ be a sublattice of $E$ and $I(F)$ be the ideal generated by $F$ in $E$. Then $F$ has $b$-property in $I(F)$.
Having $b$-property is transitive: if $E\subseteq F \subseteq G$ are sublattices of a VL $X$, then $E$ 
has $b$-property in $F$, and if $F$ has $b$-property in $G$,  then $E$ has $b$-property in $G$. 
If  $E$ has $b$-property in $G$,  then $E$ has $b$-property in every sublattice of $G$ containing $E$ as a sublattice.}
\end{example}

\bigskip

\begin{example}\label{countable b-prop} {\em
Let $(E, \| . \|)$ be a Banach lattice with order continuous norm and  $F \subseteq E$ be a norm-closed sublattice.
Let $x_n$ be a $b$-bounded sequence in $F$ such that $0 \le x_n \le e$ for some $e\in E$. Then $x_n$ is norm-Cauchy 
and is convergent to some $x\in E$. As $F$ is norm-closed, $x\in F$ and consequently $x_n \uparrow x$. 
That is to say $F$ has countable $b$-property in $E$. 
Order continuity of the ambient space is essential in this example, if one takes $E= l^\infty$ and $F=c_0$, 
Then by considering the sequence $e_n$ in $c_0$, we see that $c_0$ has no $b$-property in $l^\infty$. }
\end{example}

\bigskip

\begin{example}\label{general countable b-prop} {\em
Generalizing Example \ref{countable b-prop}, let $E(\tau)$ be LSVL with Lebesgue property.
Then every $\tau$-closed order ideal has $b$-property in $E(\tau)$. This is because every $\tau$-closed ideal 
is a band and, as  $E(\tau)$ is Dedekind complete, it is a projection band.}
\end{example}

\bigskip

\begin{example}\label{lateral completion} {\em
Given a LSVL $E(\tau)$, let us denote by $E^\lambda$ its lateral completion and $E^u$ its universal completion. 
Then the equality $(E^\lambda)^\delta = (E^\delta)^\lambda = E^u$ (see \cite[Exer.10 on p.213]{AB1}) 
shows that each laterally complete $E(\tau)$ has $b$-property in its universal completion.}
\end{example}

\bigskip

\begin{example}\label{lateral complete LSVL} {\em
If $E(\tau)$ is a laterally complete LSVL, then it has the projection property and every band on $E$ has $b$-property.
Furthermore, a subset $A \subset E^+$ of a laterally complete VL $E$ is order bounded in $E^u$ iff
it is order bounded in $E$ by \cite[Theorem~7.14]{AB1}}
\end{example}

\bigskip

Let us observe that all Lebesgue topologies on a LSVL $E(\tau)$ induce the same topology on order bounded subsets of $E$.
Therefore, if $F$ is a sublattice of $E$ then on all subsets of $F$ with $b$-property in $E$ 
all Lebesgue topologies on $E$ induce the same topology.

\bigskip

\begin{example}\label{order dense} {\em
Let $F$ be an order dense sublattice of a vector lattice $E$. If $F$ is laterally complete in its own right, 
then $F$ majorizes $E$ and therefore has $b$-property in $E$.}
\end{example}

\bigskip

We refer to \cite{AB1, MN} for all undefined terms.

%\newpage
\section{Main results}

\begin{lemma}\label{b-bounded subset is t-bounded}
Let $F$ be a sublattice of a LSVL $E(\tau)$. Then each $b$-bounded subset $B$ of $F$ is 
$\tau$-bounded subset with respect to induced topology on $F$.
\end{lemma}

\begin{proof}\
To say that $B$ is $b$-bounded is to say that $B$ is order bounded in $E$. So, if $U$ is a neighborhood of $0$ in $\tau$ then $B\subseteq \lambda U$
for some $\lambda > 0$. Then $B \subseteq \lambda U \cap F = \lambda(U \cap F)$.
\end{proof}

\bigskip

\begin{lemma}\label{order dense subl}
Let $E$ be  a vector lattice and $F$ be an order dense sublattice of $E$. Then TFAE:
\begin{enumerate}
\item[i$)$]  $F$ has $b$-property in $E$;
\item[ii$)$]  $F$ is majorizing in $E$.
\end{enumerate}
\end{lemma}

\begin{proof}\
$i)\Longrightarrow ii)$: \ 
Let $0\le x \in E$ be arbitrary, as $F$ is order dense in $E$, there exists a net $x_\alpha$  in $F$ such that $0\le x_\alpha \uparrow x$. 
As $x_\alpha$ is $b$-bounded by assumption, there exists $x_0 \in F^+$ with $0\le x_\alpha \le x_0$ for all $\alpha$, as  $x_\alpha \uparrow x$,
we have $x\le x_0$ and $F$ is majorizing.

$ii)\Longrightarrow i)$: \ 
Let $x_\alpha$ be a net in $F$ with $0\le x_\alpha \uparrow \le x$ for some $x\in E$. Since $F$ is assumed to be majorizing $E$, 
there exists $y\in F$ with $x\le y$. Consequently, $0\le x_\alpha \le y$ and $F$ has $b$-property in $E$.
\end{proof}

This yields: $E$ has $b$-property in $E^u$ iff $E$ is majorizing in $E^u$. We also have if $E(\tau)$ is a LSVL where $E$ is an ideal of 
$\hat{E}(\hat{\tau})$, where $\hat{E}$ is the completion. Then $E$ has $b$-property in  $\hat{E}(\hat{\tau})$.

On the other hand, if $E(\tau)$ is a LSVL with Fatou property, then every increasing $\tau$-bounded net of $E^+$ is order bounded in $E^u$ i.e. every 
increasing  $\tau$-bounded net of $E^+$ is $b$-bounded in $E^u$ by \cite[Theorem 7.51]{AB1}

The following property was introduced in \cite{La84} and \cite{La85}.

\bigskip

\begin{definition}\label{BOB} {\em
A locally solid vector lattice $E(\tau)$ is called {\em boundedly order bounded (BOB)} if every $\tau$-bounded net in $E^+$ is order bounded in $E$. }
\end{definition}

\bigskip

We show BOB is equivalent to $b$-property if the LSVL $E(\tau)$  has Fatou property.

\bigskip

\begin{lemma}\label{BOB=b}
Let $E(\tau)$  be a LSVL with Fatou property. Then $E$ has $b$-property in $E^u$ iff $E$ is BOB.
\end{lemma}

\begin{proof}\
Suppose $E$ is BOB and $x_\alpha$ be a net in $E$ with $0\le x_\alpha \uparrow \le x_0$ for some $x_0 \in E^u$. 
Then, by Lemma~\ref{b-bounded subset is t-bounded}, $x_\alpha$ is $\tau$-bounded in $E$ and, 
by assumption that $E$ is BOB,  $0\le x_\alpha \le x$ for some $x\in E$.

Conversely, suppose $x_\alpha$ is $\tau$-bounded increasing net in $E^+$, then by \cite[Theorem 7.50]{AB1}, 
$x_\alpha$ is order bounded in $E^u$. Thus by $b$-property of $E$ in $E^u$, there exists $x\in E$ with 
$0\le x_\alpha \le x$ and  $E(\tau)$ is BOB.
\end{proof}

\bigskip

\cite[Theorem 7.49]{AB1} shows that, in a laterally $\sigma$-complete LSVL  $E(\tau)$, every disjoint sequence in $E^+$ converges to zero 
with respect to any LS topology on $E$. We show a similar result. The proof is similar.

\bigskip

\begin{proposition}
Let $E(\tau)$  be a LSVL which has countable $b$-property in its $\sigma$-lateral completion. Then every disjoint sequence in $E^+$
converges to zero with respect to any locally solid topology on $E$. In particular, every locally solid topology on $E$ has the pre-Lebesgue property.
\end{proposition}

\begin{proof}\
Let  $x_n$ be a disjoint sequence in $E^+$. Then $n x_n$ is also a disjoint sequence in $E^+$. Then 
$x= \bigvee_{n=1}^{\infty}nx_n$ exists in the  $\sigma$-lateral completion, and we have $0 \le x_n \le \frac{1}{n}x$ for all $n$.
Countable $b$-property of $E$ in its lateral completion yields a vector $e\in E$ with $0 \le x_n \le \frac{1}{n}e$ for all $n$.
Thus $x_n$ converges to zero with respect to any locally solid topology on $E$.
\end{proof}

\bigskip

\begin{corollary}
Let $E(\tau)$  be a LSVL with Lebesgue property. If $E$ has countable $b$-property in its $\sigma$-lateral completion then the topological 
completion $\hat{E}$ of $E(\tau)$ is $E^u$.
\end{corollary}

\begin{proof}\
It follows from \cite[Theorem 7.51]{AB1}. 
\end{proof}

\bigskip

\begin{proposition}
A laterally complete vector lattice $E$ has $b$-property in every vector lattice which contains $E$ as an order dense sublattice.
\end{proposition}

\begin{proof}\
In this case $E$ majorizes the vector lattice that contains it. The result now follows from   \cite[Theorem 7.15]{AB1}. 
\end{proof}

\bigskip

In \cite[Proposition 2.22]{TM18} it is proved that if  $E(\tau)$ is a LSVL with Lebesgue topology, then a sublattice $F$ of $E$ is $u\tau$-closed
in $E$ iff it is $\tau$-closed. It was asked in \cite[Question 2.24]{TM18} whether Lebesgue assumption could be removed. The next result yields 
an answer utilizing $b$-property. 

\bigskip

\begin{proposition}
Let $F$  be an order ideal of a LSVL $E(\tau)$. If $F$ has  $b$-property in  $E$, then $F$ is $u\tau$-closed iff it is $\tau$-closed in $E$.
\end{proposition}

\begin{proof}\
As $u\tau$ is coarser than $\tau$, the forward implication is clear.

Now, suppose $F$ is $\tau$-closed and $y_\alpha$ is a net in $F$ with $y_\alpha \stackrel{u\tau}{\to} x$ for some $x\in E$. We will show $x\in F$.
The lattice operations are $u\tau$-continuous, so that $y^\pm_\alpha \stackrel{u\tau}{\to} x$. Therefore, WLOG we may assume $0 \le y_\alpha$ for
all $\alpha$. Let $z \in E^+$ be arbitrary, then
$$
  |y_\alpha \wedge z - x\wedge z | \le |y_\alpha - x | \wedge z \stackrel{\tau}{\to} 0.
$$
Since $0 \le y_\alpha \wedge x \le y_\alpha$ for all $\alpha$, and $F$ is an order ideal, we have $y_\alpha \wedge x \in F$ for all $\alpha$ and 
$y_\alpha \wedge x  \stackrel{\tau}{\to} x\wedge x$. 

Take $y \in F$, then $y_\alpha \wedge y  \stackrel{\tau}{\to} x\wedge y$, since $F$ is $\tau$-closed we have $x\wedge y \in F$ for each $y\in F^+$.
If $z\in F^d$, then $y_\alpha \wedge z =0$  for all $\alpha$ and we have $x \wedge z =0$. Thus $x\in F^{dd}$. 
That is, $x$ is in the band generated by $F$ in $E$. Hence there exists a net $z_\beta$ in $F^+$ such that $0 \le z_\beta \uparrow |x|$. 
Therefore $z_\beta$ is $b$-bounded in $E$, by $b$-property of $F$ in $E$, $0 \le z_\beta \le x_0$ for some $x_0 \in F$ and $|x| \le x_0$. 
Hence $x\in F$ as $F$ is an ideal.
\end{proof}

\bigskip

It shown in \cite[Theorem 7.39]{AB1} that a Dedekind complete vector lattice is universally complete iff it is universally $\sigma$-complete and has a weak unit.
In the next result, we replace universally $\sigma$-completeness with countable $b$-property of $E$ in $E^u$.

\bigskip

\begin{proposition}
Let $E$  be a Dedekind complete vector lattice with countable $b$-property in $E^u$ and  a weak order unit. Then $E=E^u$.
\end{proposition}

\begin{proof}\
If $E=E^u$ then $E$ has $b$-property in $E^u$ and  has a weak unit. Now we prove the converse. Let $0 < e$ be a weak order unit for $E$.
Then $E$ is an order ideal in $E^u$ by \cite[Theorem 1.40]{AB1}. Let $0 < u \in E^u$ be arbitrary. Since $e$ is also a weak unit for $E^u$ 
($E$ is order dense in $E^u$), we have $0 < u\wedge ne \uparrow u$. As $u\wedge ne \in E$ for each $n$, we see that the sequence $u\wedge ne$
is $b$-bounded in $E^u$. Therefore the sequence $u\wedge ne$ has an upper bound in $E$ by assumption. Thus $0 \le  u\wedge ne \le x$ for some
$x\in E$. As $E$ is an order ideal  in $E^u$, we have $u\in E$. 
\end{proof}

\bigskip

It is well known that if $E(\tau)$ is a LSVL with Levi property and $\tau$-complete order intervals, then $E$ is complete. In the following we reach
to the same conclusion by replacing Levi property with weaker condition that $E$ having $b$-property in  $\hat{E}(\hat{\tau})$.

\bigskip

\begin{proposition}
Let $E(\tau)$ be a LSVL with $\tau$-complete order intervals. If $E(\tau)$ has $b$-property in $\hat{E}$, then $E$ is complete.
\end{proposition}

\begin{proof}\
The assumption on order intervals implies that $E$ is an order ideal of $\hat{E}$ by \cite[Theorem 2.42]{AB1}. 
Let $0 < \hat{x} \in \hat{E}$ be arbitrary. Since $E$ is order dense in $\hat{E}$, there exists a net $x_\alpha$ such that 
$0 \le x_\alpha \uparrow  \hat{x}$. By the $b$-property of $E$ in $\hat{E}$, we can find $x_0 \in E$ with $0 \le x_\alpha \le x_0$.
But then since $x_\alpha \uparrow x$, we have $x \le x_0$ and $x\in E$. 
\end{proof}

\bigskip

\begin{proposition}
Let $F$ be a sublattice of an order complete vector lattice $E$. Suppose $F$ is order dense and majorizing in $E$. Then 
each increasing  $b$-bounded net in $F$ is $uo$-Cauchy in $F$.
\end{proposition}

\begin{proof}\
Let $x_\alpha$ be a $b$-bounded net in $F$ so that $0 \le x_\alpha \uparrow  \le e$ for some $e\in E^+$. Since $E$ is order complete, 
$x_\alpha \uparrow x$ for some $x\in E^+$.  Then $x_\alpha$ is order convergent in $E$ and hence $o$-Cauchy in $E$, thus
$x_\alpha$ is $uo$-Cauchy in $F$ by  \cite[Theorem 2.3]{GTX}
\end{proof}

\bigskip

It was observed that in \cite[Theorem 3.2]{GTX} for a net $x_\alpha$ in a regular sublattice $F$ of a vector lattice $E$, 
$x_\alpha \stackrel{uo}{\to}0$ in $F$ iff $x_\alpha \stackrel{uo}{\to}0$ in $E$. However this may fail for $u\tau$-convergence. 
$u\tau$-Convergence in a sublattice may not imply $u\tau$-convergence in the entire space. For example, 
the standard unit vectors $e_n$ in $l^\infty$ is easily seen to be a null sequence in the unbounded norm topology of
$c_0$ but not so in $l^\infty$. 

\bigskip

\begin{proposition}
Let $F$ be a sublattice of a LSVL $E(\tau)$. Suppose $F$ has $b$-property in $E$. For a net $x_\alpha$ in $F$ for which 
$x_\alpha \stackrel{u\tau}{\to}0$ in $F$, we have  $x_\alpha \stackrel{u\tau}{\to}0$ in $E(\tau)$.
\end{proposition}

\begin{proof}\
Suppose $x_\alpha \stackrel{u\tau}{\to}0$ in $F$. WLOG we may suppose $0\le x_\alpha$ for all $\alpha$. 
Then $0 \le x_\alpha \wedge y \stackrel{\tau}{\to}0$ for each $y \in F^+$. On the other hand, for each $x\in E^+$, 
$0 \le x_\alpha \wedge x \le x$ and the net $0 \le (x_\alpha \wedge x)$ is $b$-bounded in $F$, by the hypothesis,
there exists $y \in F^+$ such that $0 \le x_\alpha \wedge x \le y$ for all $\alpha$. Then
$$
  0 \le x_\alpha \wedge x \le x_\alpha \wedge y \stackrel{\tau}{\to}0
$$
from which we obtain, $x_\alpha \wedge x  \stackrel{\tau}{\to}0$. As $x$ is arbitrary $x_\alpha \stackrel{u\tau}{\to}0$ in $E(\tau)$.
\end{proof}

\bigskip

\begin{proposition}
Let $E(\tau)$ be a laterally complete vector lattice, then $E$ has $b$-property in $(E^\sim)^\sim_n$.
\end{proposition}

\begin{proof}\
Recall that $E$ is order dense in $(E^\sim)^\sim_n$. Then $E$ is majorizing in $(E^\sim)^\sim_n$ by  \cite[Theorem 7.15]{AB1}. Therefore 
$E$ has $b$-property in $(E^\sim)^\sim_n$
\end{proof}

\bigskip

\begin{proposition}
Let $E(\tau)$ be a LSVL with Lebesgue property. Then every order closed sublattice $F$ of $E(\tau)$  has 
countable $b$-property in $\hat{E}(\hat{\tau})$.
\end{proposition}

\begin{proof}\
Let $x_n$ be a $b$-bounded sequence in $E$. Then there exists $\hat{x} \in \hat{E}$ with $0 \le x_n \uparrow \hat{x}$.
Since $E(\tau)$ is assumed to have Lebesgue property, it has the $\sigma$-Lebesgue property as well as the Fatou property
by \cite[Theorem 4.8]{AB1}. Since the topology $\hat{\tau}$ of $\hat{E}$ is also Lebesgue, the sequence  $x_n$ is
$\hat{\tau}$-Cauchy in  $\hat{E}$. Then $x_n \stackrel{\hat{\tau}}{\to}x$ for some  $x\in \hat{E}$. Since $\tau$ is Fatou and 
$F$ being order closed is $\tau$-closed by \cite[Theorem 4.20]{AB1}. Thus $x\in F$. As $x_n \uparrow$, 
$x_n \stackrel{\tau}{\to}x$, hence $x = \sup x_n $, and $F$ has $b$-property in $\hat{E}$.
\end{proof}

\bigskip

\begin{proposition}
Let $F$ be a $uo$-closed sublattice of a Dedekind complete vector lattice $E$. Then $F$ has $b$-property in $E$.
\end{proposition}

\begin{proof}\
Let $x_\alpha$ be a net in $F$ with  $0 \le x_\alpha \uparrow x$ for some $x \in E$. As $E$ is Dedekind complete,
$x_\alpha \uparrow \hat{x}$ for some $\hat{x} \in E$. Then $x_\alpha \stackrel{o}{\to}\hat{x}$, consequently  
$x_\alpha \stackrel{uo}{\to}\hat{x}$ in $E$ as $F$ is $uo$-complete, $\hat{x} \in F$.
\end{proof}

\bigskip

\begin{proposition}
Let $E$ be a vector lattice admitting a minimal topology $\tau$. Let $x_n$ be a $b$-bounded sequence in $E^u$.
Then $x_n$ is $\tau$-Cauchy in $E$.
\end{proposition}

\begin{proof}\
Let $x_n$ be such that  $0 \le x_n \uparrow x^u$ for some $x^u \in E^u$.
Since $E^u$ is Dedekind complete,  $x_n$ being order bounded in $E^u$, has a supremum in $E^u$, let it be $x$.
Therefore  $x_n \stackrel{o}{\to}x$, it follows that $x_n$ is $uo$-Cauchy in $E^u$. 
Since $E$ is order dense in $E^u$, and order dense sublattices are regular, $E$ is regular in $E^u$ and by
 \cite[Theorem 3.2]{GTX}, $x_n$ is $uo$-Cauchy in $E$. As every minimal topology is Lebesgue, $\tau$ is Lebesgue
and $x_n$ is $u\tau$-Cauchy. As $\tau$ is unbounded, it follows that $x_n$ is $\tau$-Cauchy on $E$.
\end{proof}

\bigskip

\begin{definition}\label{boundedly uo-complete LSVL} {\em
A locally solid vector lattice $E(\tau)$ is called {\em boundedly $uo$-complete} if every 
$\tau$-bounded $uo$-Cauchy net in $E$ is $uo$-convergent. }
\end{definition}

\bigskip

\begin{proposition}
A boundedly $uo$-complete LSVL  $E(\tau)$ has $b$-property in $E^u$.
\end{proposition}

\begin{proof}\
Let  $0 \le x_\alpha\uparrow x^u$, where $x^u \in E^u$, be a net in $E$. As $x_\alpha$ is a $b$-bounded subset of $E$,
it is $\tau$-bounded by Lemma~\ref{b-bounded subset is t-bounded}. We show $x_\alpha$ has an upper bound in $E$. 
As $E^u$  is Dedekind complete, $\sup x_\alpha$ exists in  $E^u$. Let this supremum be $x$. Then 
$0 \le x_\alpha\uparrow x$ in $E^u$. Thus $x_\alpha \stackrel{o}{\to}x$. It follows that $x_\alpha$ is $uo$-Cauchy in $E$
as $E$ is order dense and a regular sublattice of $E^u$. Thus $x_\alpha$ being $uo$-Cauchy and $\tau$-bounded
converges to some $x'\in E$. But as $x_\alpha \stackrel{o}{\to}x$ we must have $x=x'$
\end{proof}

%\bigskip
%\begin{definition}\label{boundedly uo-complete BL} {\em
%A Banach lattice is called {\em boundedly $uo$-complete} if every 
%norm bounded $uo$-Cauchy net in $E$ is $uo$-convergent.}
%\end{definition}
\bigskip

\begin{definition}\label{monotonically complete BL} {\em
A Banach lattice is {\em  monotonically complete (the Levy property)} if every norm bounded increasing
net in $E^+$ has supremum. }
\end{definition}

\bigskip

We now show that every  boundedly $uo$-complete Banach lattice $E$ has $b$-property in $(E^\sim_n)^\sim_n$.
The proof uses an idea  of \cite{GLX18pp} in that $(E^\sim_n)^\sim_n$ is monotonically complete and the canonical map 
$J : E \to (E^\sim_n)^\sim_n$ maps a bounded increasing net in $E^+$ to a net in $(E^\sim_n)^\sim_n$ with
similar properties.

\bigskip

\begin{proposition}
Let $E$ be a boundedly $uo$-complete Banach lattice with $E^\sim_n$ separating points of $E$. 
If $0 \le x_\alpha\uparrow$ is an increasing  net in $E^+$ which is order bounded in $(E^\sim_n)^\sim_n$,
then $x_\alpha$ has an upper bound in $E$.
\end{proposition}

\begin{proof}\
Since the net $x_\alpha$ is order bounded in $(E^\sim_n)^\sim_n$, it is norm bounded in 
$(E^\sim_n)^\sim_n$ and hence norm bounded in $E$ by Lemma~\ref{b-bounded subset is t-bounded}. 

Let  $J : E \to (E^\sim_n)^\sim_n$ be the natural embedding, where $J(x)(f) = f(x)$ for each $x\in E$
and $f\in E^\sim_n$. The map $J$ is a  vector lattice isomorphism and the range  
$J (E)$ in $(E^\sim_n)^\sim_n$ is order dense  in $(E^\sim_n)^\sim_n$ by \cite[Theorem 1.43]{AB1}. 
Therefore, $J(E)$ is a regular sublattice of $(E^\sim_n)^\sim_n$.

By  \cite[2.4.19]{MN}, $(E^\sim_n)^\sim_n$ is a monotonically complete Banach lattice. 
Thus, the increasing net $J(x_\alpha)$ has a supremum in $(E^\sim_n)^\sim_n$ say $x$. 

So $ J (x_\alpha)\uparrow x$ and $J (x_\alpha)$ is order Cauchy in $(E^\sim_n)^\sim_n$. 
It follows that $J (x_\alpha)$ is $uo$-Cauchy in $(E^\sim_n)^\sim_n$ and in the regular sublattice $J(E)$. 
As $J$ is $1$-$1$ and onto $J (E)$ is lattice isomorphism, $x_\alpha$ is  $uo$-Cauchy in $E$. 
$E$ is boundedly $uo$-complete, $x_\alpha\stackrel{uo}{\to}x_1$  for some $x_1 \in E$.
On the other hand $0 \le x_\alpha\uparrow$, thus $x_\alpha\uparrow x_1$ and $x_\alpha$
is order bounded in $E$.
\end{proof}

%\newpage


{\small 
\begin{thebibliography}{30}
\bibitem{AB1}
Aliprantis, C.D., Burkinshaw, O.: 
Locally Solid Riesz Spaces with Applications to Economics, 2nd edition. 
American Mathematical Society, Providence, RI (2003)

\bibitem{AAT03}
Alpay, \c{S}.,  Alt{\i}n, B.,  Tonyali, C.:
On property $b$ of vector lattices, 
Positivity, 7, 135-139 (2003),

\bibitem{AAT06}
Alpay, \c{S}.,  Alt{\i}n, B.,  Tonyali, C.:
A note on Riesz spaces with property $b$, 
Czechoslovak Math. J. 56 (131), 765-772 (2006)

\bibitem{AEmG}
Alpay, \c{S}.,  Emelyanov, E., Gorokhova, S.:
Bibounded $uo$-convergence and $bbuo$-duals of vector lattices,
https://arxiv.org/abs/2009.07401v1 (2020)

\bibitem{AEr09}
Alpay, \c{S}., Ercan, Z.:
Characterizations of Riesz spaces with $b$-property,  
Positivity, 13 no. 1, 21-30 (2009)

%\bibitem{GLX}
%Gao, N., Leung, D., Xanthos, F.: 
%Duality for unbounded order convergence and applications,
%Positivity, 22, 711-725 (2018) 

\bibitem{GLX18pp}
Gao, N., Leung, D., Xanthos, F.: 
Dual representation of risk measures on Orlicz spaces. 
Preprint: arXiv (2018) 

\bibitem{GTX}
Gao, N., Troitsky, V., Xanthos, F.:
$Uo$-convergence and its applications to Ces{\'a}ro means in Banach lattices.
Isr. J. Math. 220, 649-689 (2017)

\bibitem{La84}
Labuda, I.: 
Completeness type properties of locally solid Riesz spaces,
Studia Math. 77, 349-372 (1984)

\bibitem{La85}
Labuda, I.: 
On boundedly order-complete locally solid Riesz spaces,
Studia Math. 81, 245-258 (1985)

%[Lab87] I. Labuda, Submeasures and locally solid topologies on Riesz spaces, Math. Z. 195 (1987), 179-196.

\bibitem{MN}
Meyer-Nieberg, P.: 
Banach Lattices. 
Universitext, Springer-Verlag, Berlin (1991)

\bibitem{TM18}
Taylor, M.A.: 
Unbounded Convergences in Vector Lattices, Master's thesis, 
University of Alberta, (2018)

%\bibitem{TM19}
%Taylor, M.A.: 
%Unbounded topologies and $uo$-convergence in locally solid vector lattices
%J. Math. Anal. Appl. 472  981-1000 (2019) 

\end{thebibliography}
}
\end{document}